\documentclass[12pt]{article}
\usepackage{srcltx}
\usepackage{amsfonts,amssymb,mathrsfs,amsmath,cmtiup}
\usepackage{amsthm}

\newtheorem{theorem}{Theorem}[section]
\newtheorem{lemma}[theorem]{Lemma}
\newtheorem{proposition}[theorem]{Proposition}
\newtheorem{corollary}[theorem]{Corollary}

\theoremstyle{definition}
\newtheorem*{definition}
{Definition}

\newtheorem*{Index Convention}{Index Convention}
\newtheorem*{notation}{Notation}

\def\keywords#1{\par\medskip
\noindent\textbf{Keywords.} #1}

\def\subjclass#1{{\renewcommand{\thefootnote}{}%
\footnote{\emph{Mathematics Subject Classification (2010):} #1}}}

\begin{document}
\let\le=\leqslant
\let\ge=\geqslant
\let\leq=\leqslant
\let\geq=\geqslant
\newcommand{\e}{\varepsilon }
\newcommand{\f}{\varphi }
\newcommand{ \g}{\gamma}
\newcommand{\F}{{\Bbb F}}
\newcommand{\N}{{\Bbb N}}
\newcommand{\Z}{{\Bbb Z}}
\newcommand{\Q}{{\Bbb Q}}
\newcommand{\C}{{\Bbb C}}
\newcommand{\R}{\Rightarrow }
\newcommand{\W}{\Omega }
\newcommand{\w}{\omega }
\newcommand{\s}{\sigma }
\newcommand{\hs}{\hskip0.2ex }
\newcommand{\ep}{\makebox[1em]{}\nobreak\hfill $\square$\vskip2ex }
\newcommand{\Lr}{\Leftrightarrow }

\title{Lie algebras with Frobenius dihedral groups of automorphisms}

\markright{}

\author{{\sc N.\,Yu.~Makarenko}\\ \small Sobolev Institute of Mathematics, Novosibirsk, 630\,090,
Russia }

\author{
{N.\,Yu.~Makarenko\footnote{The work is supported by  Russian
Science Foundation, project 21-11-00286}}\\
\small  Sobolev Institute of Mathematics, Novosibirsk, 630\,090,
Russia
\\[-1ex] \small  natalia\_makarenko@yahoo.fr}
\date{}
\maketitle \subjclass{Primary 17B40, 17B30, 17B70; Secondary
17B65}

\begin{abstract}
 Suppose that a Lie algebra $L$  admits a finite Frobenius group of automorphisms $FH$ with
cyclic kernel $F$   and complement $H$ of order 2, such that the
fixed-point subalgebra of $F$ is trivial and the fixed-point
subalgebra of $H$ is metabelian.  Then the  derived length of $L$
is bounded by a constant.
\end{abstract}


\section{Introduction}

Let $L$ be a Lie algebra, $G$ its group of automorphisms. We
denote the fixed-point subalgebra of $G$ in $L$ by $$C_L(G)=\{l\in
L \,\,\,\mid \,\,\, l^{\varphi}=l \text{ for all } \varphi\in G\}.
$$

 Suppose
that  $L$ admits a Frobenius group of automorphisms $FH$ with
kernel $F$ and complement $H$ of order $q$ such that the
fixed-point subalgebra of $F$ is trivial, i.e. $C_L(F)=0$. In
\cite{khu-ma-shu} it was proved, that if the fixed-point
subalgebra  $C_L(H)$ of $H$ is nilpotent of class~$c$, then $L$ is
nilpotent of class bounded by a function depending only on $q$ and
$c$.

\vskip1ex
 In this paper we make the first attempt to extend this
result to the case of  solvability of $C_L(H)$ (instead of
nilpotency). We consider a Frobenius dihedral group acting on a
Lie algebra in such a way that $C_L(F)=0$ and  $C_L(H)$  is
metabelian, and prove the solvability of bounded length of the
whole algebra.

\begin{theorem}\label{th1}
 Suppose that a Lie algebra $L$  of arbitrary   dimension admits a finite Frobenius group of automorphisms $FH$ with
cyclic kernel $F$  and complement $H$ of order 2, such that the
fixed-point subalgebra of $F$ is trivial and the fixed-point
subalgebra of $H$ is metabelian.  Then the  derived length of $L$
is bounded by a constant.
\end{theorem}

The proof is based on combinatorial calculus in a $\Z/n\Z$-graded
Lie algebra for $n$ being the order of the Frobenius kernel $F$.
Such a $\Z/n\Z$-graded algebra arises naturally, as we extend the
ground field by a primitive $n$th root $\omega$ of 1 and consider
the eigenspaces $L_i=\{x\in L \mid x^{\varphi}=\omega^i x\}$ for
the eigenvalues $\omega^i$ of the automorphism $\varphi$, a
generator of $F$. These subspaces satisfy the properties:
$$[L_i, L_j]\subseteq L_{i+j\,(\rm{mod}\,n)}\qquad \text{and}\qquad L= \bigoplus
_{i=0}^{n-1}L_i,$$  and thus form  a $(\Bbb Z /n\Bbb Z )$-grading.
It is clear that  $L_0=C_L(F)$. The assumptions of Theorem
\ref{th1} can be translated into the language of graded algebras
by introducing two notions: so called ``(-1)-independent
sequence'' and  the ``selective metabelian condition''.  We call a
sequence $(a_1, a_2, a_3, a_4)$, $a_i\in\Z/n\Z$,
\textit{$(-1)$-independent}, if for any $(t_1, t_2,t_3, t_4)$,
$t_i\in\{0,1\}$,
$$t_1a_{1}+t_2a_{2}+t_3a_{3}+t_4a_{4}=0 \,\,\,\,\,\Rightarrow \,\,\,\, (t_1, t_2,t_3,
t_4)= (0,0,0, 0). $$ We say that a $(\Z/n\Z)$-graded Lie algebra
$L$ satisfies the \textit{selective metabelian condition} if
$$
\big[[L_{d_1},L_{d_2}],[L_{d_{3}}, L_{d_{4}}]\big]=0\quad \text{
whenever } (d_1,d_2, d_3,d_{4}) \text{ is $(-1)$-independent}.
$$

To demonstrate Theorem \ref{th1} it suffices to prove the
solvability of bounded length of a $(\Z/n\Z)$-graded Lie algebra
with $L_0=0$ satisfying selective metabelian condition (see
section \ref{section-reduction}).

\vskip1ex The general idea  and many arguments  are borrowed from
\cite{khu-ma-shu} and \cite{ma-shu}. In particular, in
\cite{khu-ma-shu} and \cite{ma-shu} there were already obtained
some analogues of Lemmas  \ref{l4}, \ref{l5}, \ref{l7} and
Corollary \ref{c6}.


\section{Preliminaries}

If $M,\, N$ are subspaces of
   $L$, then $[M,N]$
denotes the subspace, generated  by all the products $[m,n]$ for
$m\in M$, $n\in N$. If $M$ and $N$ are ideals, then $[M,N]$ is
also an ideal; if $H$ is a (sub)algebra, then $[H,H]$ is  an ideal
of $H$ and, in particular, its subalgebra.


\vskip1ex
 A simple product  $[a_1,a_2,a_3,\ldots, a_s]$ is by
definition the left-normalized product
$$[...[[a_1,a_2],a_3],\ldots, a_s].$$

\vskip1ex
 Note that any (complex) product in  elements in $L$  can be expressed as a linear combination of simple products of
the same length in the same elements.  Also, it follows that the
ideal in  $L$ generated by a  subspace $S$ is the subspace
generated by all the  simple products  $[y_j,x_{i_1},
x_{i_2},\ldots x_{i_t}]$, where $t\in \Bbb N$ and $x_{i_k}\in L,
y_j\in S$. In particular, if $L$ is generated by a subspace $M$,
then  its space is spanned  by simple products in elements of~$M$.

 \vskip1ex
 The derived series of an algebra $L$ is defined as
$$L^{(0)}=L, \; \; \; \, \, \, \, \, \,
L^{(i+1)}=[L^{(i)},L^{(i)}].$$ Then $L$ is solvable of derived
length at most $n$ if $L^{(n)}=0$. An algebra of derived length~2
is called \textit{metabelian}.

 \vskip1ex






  A Lie algebra $L$ over a field
is \textit{ $({\Bbb Z}/n{\Bbb Z})$-graded} if
$$L=\bigoplus_{i=0}^{n-1}L_i\qquad \text{ and }\qquad [L_i,L_j]\subseteq L_{i+j\,({\rm mod}\,n)},$$
where  $L_i$ are subspaces  of~$L$. Elements of $L_i$ are referred
to as \textit{homogeneous} and the subspaces $L_i$  are called
\textit{homogeneous components} or  \textit{grading components}.
In particular, $L_0$ is called the zero homogeneous component or
simply zero component.

 \vskip1ex

An additive subgroup $H$ of $L$ is called \textit{homogeneous} if
$H=\bigoplus_i (H\cap L_i)$ and we set $H_i=H\cap L_i$. Obviously,
any subalgebra or an ideal generated by homogeneous subspaces is
 homogeneous. A homogeneous subalgebra  can be regarded as a
$({\Bbb Z} /n{\Bbb Z} )$-graded algebra with the induced grading.
It follows that the terms of the derived series of $L$, the ideals
$L^{(k)}$  are also $({\Bbb Z} /n{\Bbb Z} )$-graded algebras with
induced grading
$$L^{(k)}_i=L^{(k)}\cap L_i.$$

\begin{Index Convention}
Henceforth a small Latin letter with an index $i\in \Bbb Z/n\Bbb
Z$ will denote a homogeneous element in the grading component
$L_i$, with the index only indicating which component this element
belongs to: $x_i\in L_i$. We will not be using numbering indices
for elements of the $L_i$, so that different elements can be
denoted by the same symbol when it only matters which component
the elements belong to. For example, $x_{i}$ and $x_{i}$ can be
different elements of $L_{i}$, so that $[x_{i},\, x_{i}]$ can be a
non-zero element of $L_{2i}$.
\end{Index Convention}

\section{Reduction to $\Z/n\Z$-graded Lie algebras
with \\metabelian selection condition}\label{section-reduction}

Let $L$ be   a Lie  algebra $L$ that satisfies the hypothesis of
Theorem \ref{th1}. Denote a generator of the Frobenius kernel $F$
by $\varphi$.  Let $n$ be the order of $\varphi$.   If the ground
field $\Bbb K$ does not contains a primitive $n$th root $\omega$
of 1, we extend $K$ by $\omega$ and denote the resulting Lie
algebra by $\widetilde L$. The group $FH$ acts  on $\widetilde L$
this action inherits  the conditions that $C_{\widetilde L}(F)=0$
and  $C_{\widetilde L}(H)$ is metabelin. Thus, we can assume that
$L=\widetilde L$ and the ground field contains~$\omega$.

\vskip1ex We consider the eigenspaces $L_i=\{x\in L \mid
x^{\varphi}=\omega^i x\}$, $i=0,\ldots, n-1$, of $\varphi$, for
the eigenvalues $\omega^i$. One can verify that
$$[L_i, L_j]\subseteq L_{i+j\,(\rm{mod}\,n)}\qquad \text{and}\qquad L= \bigoplus _{i=0}^{n-1}L_i,$$
so this is a $(\Bbb Z /n\Bbb Z )$-grading. We also have
$L_0=C_L(F)=0$.

\vskip1ex  Let $h$ be a generator  of $H$. By definition of the
Frobenius group $C_H(f) = 1$ for every non-identity $f$ in $F$. It
follows that $h$ acts on $F$ without non-trivial fixed points.
Hence,  $\varphi^{h^{-1}} = \varphi^{-1}$ and 2 does not divide
$n$.

\vskip1ex The group $H$ permutes the components $L_i$ in the
following way: $L_i^h=L_{-i}$.  In fact, if $x_i\in L_i$, then
$(x_i^h)^{\varphi}=x_i^{h\varphi
h^{-1}h}=(x_i^{\varphi^{-1}})^h=\omega^{-i}x_{i}^h.$

\vskip1ex In what follows, to simplify the notations, (under the
Index Convention) we will denote $(x_s)^{h}$ by $x_{-s}$ for
$x_s\in L_s$.

\vskip1ex Let $x_{a_1}, x_{a_2}, x_{a_3}, x_{a_4}$ be homogeneous
elements in $L_{a_1},L_{a_2}, L_{a_3},L_{a_{4}}$ respectively.
Consider the sums
\begin{align*}
X_1&=x_{a_1}+x_{-a_1},\\
\vdots&\\
X_{4}&=x_{a_{4}}+x_{-a_{4}}.
\end{align*}
Since all of them lie in subalgebra $C_L(H)$, which is metabelian,
it follows that
$$\big[[X_1,X_2],[X_3,X_4]\big]=0.$$
We expand the expressions to obtain on the left a linear
combination of products in the $x_{\pm a_i}$, which in particular
involves the term $\big[[x_{a_1},x_{a_2}],[x_{a_3},
x_{a_{4}}]\big]$. Suppose that the product
$\big[[x_{a_1},x_{a_2}],[x_{a_3}, x_{a_{4}}]\big]$ is non-zero.
Then there must be other terms in the expanded expression that
belong to the same component $L_{a_1+a_2+a_3+a_{4}}$, that is
$$a_{1}+a_{2}+a_{2}+a_{4}=\epsilon_1a_{1}+\epsilon_2a_{2}+\epsilon_3a_{3}+\epsilon_{4}a_{4}$$
for some $\epsilon_i\in\{1,-1\}$ not all of which are 1, or
equivalently (since $n$ is odd),
$$t_1a_{1}+t_2a_{2}+t_3a_{3}+t_{4}a_{4}=0$$
for some $t_i\in\{0,1\}$ not all of which are 0.

\vskip1ex This leads us to the followings definitions.

\begin{definition}
Let $a_1,\dots,a_k$ be not necessarily distinct non-zero elements
of $\Bbb Z/n\Bbb Z$. We say that the sequence $(a_1,\dots,a_k)$ is
\textit{$(-1)$-dependent} if
$$t_1a_{1}+t_2a_{2}+t_3a_{3}+t_4a_{4}=0$$
for some $t \in\{0,1\}$ not all of which are 0. If the sequence
$(a_1,\dots,a_k)$ is not $(-1)$-dependent, i.e.
 if for any $(t_1, t_2,t_3, t_4)$,
$t_i\in\{0,1\}$,
$$t_1a_{1}+t_2a_{2}+t_3a_{3}+t_4a_{4}=0 \,\,\,\,\,\Rightarrow \,\,\,\, (t_1, t_2,t_3,
t_4)= (0,0,0, 0), $$
 we call it \textit{$(-1)$-independent}.
\end{definition}


\begin{definition}
  We say that a $(\Z/n\Z)$-graded Lie  algebra $L$ satisfies
the \textit{selective metabelian condition} if, under the Index
Convention,
\begin{equation}\label{select}
\big[[x_{d_1},x_{d_2}],[x_{d_{3}}, x_{d_{4}}]\big]=0\quad \text{
whenever } (d_1,d_2. d_3,d_{4}) \text{ is $(-1)$-independent}.
\end{equation}
 \end{definition}

 To demonstrate
Theorem \ref{th1}  it suffices to prove the solvability of bounded
length of a $(\Z/n\Z)$-graded Lie algebra with trivial
zero-component satisfying the selective metabelian condition.

\section{Proof}

We begin with some technical  Lemmas.

\begin{notation}
For a given $(-1)$-independent sequence $(a_1,\dots,a_k) $ we
denote by $D(a_1,\dots,a_k)$ the set of all $j\in\Bbb Z/n\Bbb Z$
such that $(a_1,\dots,a_k,j)$ is $(-1)$-dependent.
\end{notation}

\begin{notation} For a arbitrary sequence $(b_1,\dots,b_s)$ we denote by
$\widetilde D (b_1, \ldots, b_s)$ the set of all linear
combinations of the form $u_1 b_1+\cdots+u_k b_k$ with $u_i\in
\{0,\pm 1,\pm 2\}$.
\end{notation}

\begin{lemma}\label{l1} Suppose that a $(\Z/n\Z)$-graded Lie algebra $L$ with $L_0=0$, $n$
odd, satisfies the selective metabelian condition~\eqref{select}.
Let $a_1, a_2, a_3, a_4\in \Bbb Z/n\Bbb Z$ such that $(a_1, a_2,
a_3)$ is an $(-1)$-independent sequence and $a_4\in D(a_1,
a_2,a_3)$.  Then (under the Index Convention)
$$\big[[x_{a_1},x_{a_2}],[x_{a_3},x_{a_4}],
x_b\big]=0$$ for all $x_{b}$ with $b\notin \widetilde D (a_1, a_2,
a_3, a_4)$.
\end{lemma}

\begin{proof}  Suppose by contradiction that
\begin{equation}\label{l1-f0}\big[[x_{a_1},x_{a_2}],[x_{a_3},x_{a_4}],
x_b\big]\end{equation}
 is not trivial. We rewrite \eqref{l1-f0} using the
Jacobi identity
 and anti-commutativity
$$\Big[[x_{a_1},x_{a_2}],[x_{a_3},x_{a_4}], x_b\Big]=
$$
$$\Big[\big[x_{a_1},x_{a_2},x_b\big],
\big[x_{a_3},x_{a_4}\big]\Big]+ \Big[\big[x_{a_1},x_{a_2}\big],
\big[x_{a_3},x_{a_4}, x_b\big]\Big]=$$
\begin{equation}\label{l1-f1}-\Big[\big[x_b,[x_{a_1},x_{a_2}]\big],
\big[x_{a_3},x_{a_4}\big]\Big]+
\end{equation}
\begin{equation}\label{l1-f2}\Big[\big[x_b
, [x_{a_3},x_{a_4}]\big],\big[x_{a_1},x_{a_2}\big]\Big].
\end{equation}
By the selective metabelian condition  the  summand \eqref{l1-f1}
is not trivial only if  the sequence $(b, a_1+a_2,  a_3, a_4)$ is
$(-1)$-dependent, i.e. there is a tuple $(t_1, t_2, t_3, t_4)\neq
(0,0,0,0)$ with $t_i\in \{0,1\}$, such that
$$t_1b+
t_2(a_1+a_2)+t_3 a_3+t_4 a_4=0.$$  We consider all possible
4-tuples $(t_1, t_2, t_3, t_4)$ with $t_i\in \{0,1\}$. If $t_1\neq
0$, then $b$ is a linear combination of elements $(a_1+a_2),\,
a_3, \, a_4$ with coefficients from $\{0,-1\}$. It follows that
$b\in \widetilde D (a_1, a_2, a_3, a_4)$ that contradicts the
assumption. Thus,~$t_1=0$. If among the $t_i$ only one integer is
non zero, the corresponding index and consequently the element
with this index is trivial (since $L_0=0$), that contradicts
non-triviality of \eqref{l1-f0}. The tuples  $(0,0,1,1)$, $(0, 1,
1,1)$ give accordingly the triviality of the subproducts
$[x_{a_3}, x_{a_4}]$ and
$\Big[[x_{a_1},x_{a_2}],[x_{a_3},x_{a_4}]\Big]$ in \eqref{l1-f0},
since they belong to $L_0=0$. It remains to consider only 2
following tuples: $(0,1,1,0)$, $(0,1,0,1)$. The tuple $(0,1,1,0)$
gives $a_1+a_2+a_3=0$ that is impossible since $a_1, a_2, a_3$ are
$(-1)$-independent.   Thus, \eqref{l1-f0} and~\eqref{l1-f1} are
both not trivial only if $(t_1,t_2,t_3,t_4)=(0,1,0,1)$ and
$\boldsymbol{a_4=-a_1-a_2}.$

\vskip1ex In the same way, we determine under what condition the
products \eqref{l1-f0} and \eqref{l1-f2}   are simultaneously not
 trivial. The sequence $(b, a_3+a_4, a_1, a_2)$ should be
$(-1)$-dependent. The only tuples that do not imply the triviality
of  \eqref{l1-f2} are $(0,1,0,1)$ and  $(0,1,1,0)$. This
corresponds respectively to
$$\boldsymbol{a_4=-a_2-a_3} \text{ and } \boldsymbol{a_4=-a_1-a_3}.$$
These two  conditions are actually equivalent and one can be
obtained from the other  by interchanging $a_1$ and $a_2$ due to
anti-commutativity.

\vskip1ex We will examine separately the two cases: $a_4=-a_1-a_2
$ and $a_4=-a_1-a_3.$

\vskip2ex

 {\bf Case: $\boldsymbol{a_4=-a_1-a_2}$}.

 We substitute $a_4=-a_1-a_2$
in \eqref{l1-f0},   rewrite it as the sum of \eqref{l1-f1} and
\eqref{l1-f2} and expand the inner product in \eqref{l1-f1}
 by the Jacobi identity:

$$\Big[[x_{a_1},x_{a_2}],[x_{a_3},x_{-a_1-a_2}], x_b\Big]=$$
$$-\Big[\big[x_b,[x_{a_1},x_{a_2}]\big],
\big[x_{a_3},x_{-a_1-a_2}\big]\Big]+ \Big[\big[x_b ,
[x_{a_3},x_{-a_1-a_2}]\big],\big[x_{a_1},x_{a_2}\big]\Big]=$$
\begin{equation}\label{l1-f4}=-\Big[\big[\boldsymbol{ x_b}, x_{a_1},x_{a_2}\big],
\big[x_{a_3},x_{-a_1-a_2}\big]\Big]+\end{equation}
\begin{equation}\label{l1-f5} \Big[\big[\boldsymbol{ x_b},x_{a_2},x_{a_1}\big],
\big[x_{a_3},x_{-a_1-a_2}\big]\Big]+\end{equation}
\begin{equation}\label{l1-f6}\Big[\big[x_b ,
[x_{a_3},x_{-a_1-a_2}]\big],\big[x_{a_1},x_{a_2}\big]\Big].\end{equation}
The product \eqref{l1-f0}  is not trivial only if at least one of
the following 3 sequences of indices is $(-1)$-dependent:
\begin{equation}\label{l1-f8}\big((b+a_1), a_2,a_3,(-a_1-a_2)\big),\end{equation}
\begin{equation}\label{l1-f9} \big((b+a_2), a_1, a_3, (-a_1-a_2)\big),\end{equation}
\begin{equation}\label{l1-f11} \big(b, (a_3-a_1-a_2), a_1,
a_2)\big).
\end{equation}
 We consider all possible tuples $(t_1, t_2,
t_3, t_4)$ with $t_i\in \{0,1\}$ and investigate the corresponding
equations on the indices. If $t_1\neq 0$, then $b$ is a linear
combination of $a_1$, $a_2$, $a_3$ with coefficients belonging to
$\{ 1,-1 \}$ that contradicts the hypothesis of the lemma. Suppose
that $t_1=0$. If among the $t_i$ only one integer is not zero, the
product is trivial, since the corresponding element in the
products \eqref{l1-f4}, \eqref{l1-f5}, \eqref{l1-f6} belongs to
$L_0=0$. If $(t_1, t_2, t_3, t_4)=(0,0,1,1)$ then the products
equals to 0, since the sum of the indices of the second subproduct
is 0. It remains to consider the tuples $(0,1,1,0), (0,1,0,1),
(0,1,1,1)$.

\vskip1ex  Let  $(t_1, t_2, t_3, t_4)=(0,1,1,0)$. Then $a_2+a_3=0$
in  \eqref{l1-f8}; $a_1+a_3=0$ in \eqref{l1-f9},
$\boldsymbol{a_3=a_2}$ in~\eqref{l1-f11}.
 The first two equations are impossible,
 since $a_1$, $a_2$, $a_3$ are $(-1)$-independent.

\vskip1ex
 If
$(t_1, t_2, t_3, t_4)=(0,1,0,1)$, then
 $a_1=0$ in \eqref{l1-f8}; $a_2=0$
\eqref{l1-f9}; $\boldsymbol{a_3=a_1}$ in \eqref{l1-f11}. The first
and the second equalities imply respectively that  $x_{a_1}\in
L_0=0$ or $x_{a_2}\in L_0=0$ , and thus, \eqref{l1-f0} is trivial.

\vskip1ex For $(t_1, t_2, t_3, t_4)=(0,1,1,1)$ we have
$\boldsymbol{a_3=a_1}$ in \eqref{l1-f8}; $\boldsymbol{a_3=a_2}$ in
\eqref{l1-f9}; $a_3=0$ in \eqref{l1-f11} that corresponds to the
trivial product.

\vskip1ex Thus,  if $a_4=-a_1-a_2$, \eqref{l1-f0} is not trivial
only if $\boldsymbol{a_3=a_1}$ or $\boldsymbol{a_3=a_2}$. These
cases are equivalent as one can be obtained from the other  by
interchanging $a_1$ and $a_2$.

\vskip1ex

Let $\boldsymbol{a_2=a_3}$. In \eqref{l1-f0} we replace  $a_4$ by
$-a_1-a_2$ and $a_3$ by $a_2$, represent the obtained product as
the sum of \eqref{l1-f1} and \eqref{l1-f2} and expand the fist
subproducts in both summands by the Jacobi identity:
$$\Big[[x_{a_1}, x_{a_2}],\big[x_{a_2},x_{-a_1-a_2}\big], x_b
\Big]=$$
\begin{equation}\label{l1-f12}-\Big[[x_b, x_{a_1},
x_{a_2}],[x_{a_2},x_{-a_1-a_2}]\Big]+\end{equation}
\begin{equation}\label{l1-f13}\Big[[x_b, x_{a_2}, x_{a_1}],[x_{a_2},x_{-a_1-a_2}]\Big]+\end{equation}
\begin{equation}\label{l1-f14}\Big[[x_b, x_{a_2}, x_{-a_1-a_2}],[x_{a_1},x_{a_2}]\Big]+\end{equation}
\begin{equation}\label{l1-f15}
-\Big[[x_b, x_{-a_1-a_2},
x_{a_2}],[x_{a_1},x_{a_2}]\Big].
\end{equation}
In \eqref{l1-f13}
we can  move the element $x_{a_1}\in L_{a_i}$ to the right over
the product $[x_{a_2},x_{-a_1-a_2}]\in L_{-a_1}$, since $[L_{a_i},
L_{-a_i}]=0$. We obtain
$$
\Big[[x_b, x_{a_2},
\boldsymbol{x_{a_1}}],[x_{a_2},x_{-a_1-a_2}]\Big]=\Big[[x_b,
x_{a_2}],[x_{a_2},x_{-a_1-a_2}], \boldsymbol{x_{a_1}}\Big].
$$
In the same way, in \eqref{l1-f14} we move the element
$x_{-a_1-a_2}$ to the right over the product $[x_{a_1},x_{a_2}]\in
L_{a_1+a_2}$ to get
$$
\Big[[x_b, x_{a_2},
\boldsymbol{x_{-a_1-a_2}}],[x_{a_1},x_{a_2}]\Big]=
\Big[[x_b,
x_{a_2}],[x_{a_1},x_{a_2}], \boldsymbol{x_{-a_1-a_2}}\Big].
$$
It suffices to prove that the products \eqref{l1-f12},
\eqref{l1-f15}, $\Big[[x_b, x_{a_2}],[x_{a_2},x_{-a_1-a_2}]\Big]$,
$\Big[[x_b, x_{a_2}],[x_{a_1},x_{a_2}]\Big]$ are trivial. The
corresponding sequences of indices are:
\begin{equation}\label{l1-f18}\Big((b+a_1), a_2, a_2,(-a_1-a_2)\Big)
\end{equation}
\begin{equation}\label{l1-f19}\Big((b-a_1-a_2), a_2, a_1, a_2\Big)
\end{equation}
\begin{equation}\label{l1-f20}\Big(b, a_2,  a_2, (-a_1-a_2)\Big)
\end{equation}
\begin{equation}\label{l1-f21}\Big(b, a_2, a_1, a_2\Big).
\end{equation}
We consider  all possible 4-tuples $(t_1, t_2, t_3, t_4)$, $t_i\in
\{0, 1\}$, and the solutions of the equations
$t_1y_1+t_2y_2+t_3y_3+t_4y_4=0$ for $(y_1, y_2, y_3, y_4)$ being
successively  \eqref{l1-f18}, \eqref{l1-f19}, \eqref{l1-f20} and
\eqref{l1-f21}.

\vskip1ex If $t_1\neq 0$, then $b\in \widetilde D(a_1, a_2, a_3,
a_4)$ for all the sequences under consideration. This contradicts
the hypothesis of the lemma. Thus, $t_1=0$. As we have already
noticed above, a tuple $(t_1, t_2, t_3, t_4)$ with only one
non-zero value implies the triviality of the product, as well as
the tuple $(0,0,1,1)$. It remains to consider the tuples
$(0,1,1,0)$, $(0,1,0,1)$, $(0,1,1,1)$.

Let $(t_1, t_2, t_3, t_4)=(0,1,1,0).$ Then $2a_2=0$ in
\eqref{l1-f18} and \eqref{l1-f20}, that corresponds to the trivial
product as $n$ is odd;  $a_1+a_2=0$ in \eqref{l1-f19} and
\eqref{l1-f21}, which is impossible  since $a_1, a_2$ are
$(-1)$-independent.

\vskip1ex If $(t_1, t_2, t_3, t_4)=(0,1,0,1),$ then $-a_1=0$ in
\eqref{l1-f18} and \eqref{l1-f20};  $2a_2=0$ in \eqref{l1-f19} and
\eqref{l1-f21}. Both these conditions  imply the triviality of the
corresponding products.

\vskip1ex Let now  $(t_1, t_2, t_3, t_4)=(0,1,1,1).$ Then
$\boldsymbol{a_2=a_1}$ in \eqref{l1-f18};
$\boldsymbol{2a_2+a_1=0}$ in \eqref{l1-f19} and \eqref{l1-f21};
$-a_1=0 $ in \eqref{l1-f20}, which corresponds to the trivial
product.

\vskip1ex It follows that  either $\boldsymbol{a_1=a_2}$, or
$\boldsymbol{a_1=-2a_2}$. Hence, \eqref{l1-f0} has the initial
segment either of the form
$$\Big[[x_{a_2},x_{a_2}],[x_{a_2}, x_{-2a_2}]\Big] $$ or of the
form
$$\Big[[x_{-2a_2},x_{a_2}],[x_{a_2}, x_{a_2}]\Big].$$
The first product is trivial since
$$\Big[[x_{a_2},\boldsymbol{x_{a_2}}],[x_{a_2}, x_{-2a_2}]\Big]=
\Big[\big[x_{a_2},[x_{a_2}, x_{-2a_2}]\big],
\boldsymbol{x_{a_2}}\Big]+
\Big[x_{a_2},\big[\boldsymbol{x_{a_2}},[x_{a_2},
x_{-2a_2}]\big]\Big]
$$
and the subproducts $\big[x_{a_2},[x_{a_2}, x_{-2a_2}]\big]$,
$\big[\boldsymbol{x_{a_2}}, [x_{a_2}, x_{-2a_2}]\big]$ have zero
sum of indices. The second product is actually the same as the
first one due to anti-commutativity.

\vskip2ex {\bf Case: $\boldsymbol{a_4=-a_1-a_3}$}.

\vskip1ex We substitute $-a_1-a_3$ instead of $a_4$ in
\eqref{l1-f0}, rewrite it as the sum of \eqref{l1-f1} and
\eqref{l1-f2} and expand the first inner product in \eqref{l1-f2}
by the Jacobi identity:

\begin{equation}
\big[[x_{a_1},x_{a_2}], [x_{a_3},x_{a_4}], x_b\big]=
\big[[x_{a_1},x_{a_2}], [x_{a_3},x_{-a_1-a_3}]
x_b\big]=
\end{equation}
\begin{equation}=-\Big[\big[x_b,[x_{a_1},x_{a_2}]\big],
\big[x_{a_3},x_{-a_1-a_3}\big]\Big]\end{equation}
\begin{equation} +\big[[ x_b,x_{a_3},x_{-a_1-a_3}],
[x_{a_1},x_{a_2}]\big]
\end{equation}
\begin{equation} -\big[[ x_b,x_{-a_1-a_3}, x_{a_3}],
[x_{a_1},x_{a_2}]\big].
\end{equation}
The corresponding sequences of indices are:
\begin{equation}\label{l1-f25}
\big(b, (a_1+a_2), a_3, (-a_1-a_3)\big),
\end{equation}
\begin{equation}\label{l1-f26}
\big((b+a_3), (-a_1-a_3), a_1, a_2\big),
\end{equation}
\begin{equation}\label{l1-f27}
\big((b-a_1-a_3), a_3, a_1, a_2\big).
\end{equation}
As in the previous arguments we need to consider only the tuples
$(0,1,1,0)$, $(0,1,0,1)$, $(0,1,1,0)$, since the other tuples lead
to the triviality of  all the  products under consideration.

\vskip1ex Let $(t_1, t_2, t_3, t_4)=(0,1,1,0)$. Then
$a_1+a_2+a_3=0$ for \eqref{l1-f25}, which is impossible since
$a_1, a_2,a_3$ are $(-1)$-independent;
 $a_3=0$ for \eqref{l1-f26}, which implies the triviality of the product; $a_1+a_3=0$ for
\eqref{l1-f27}, which is impossible, since $a_1, a_3$ are
$(-1)$-independent.

\vskip1ex If $(t_1, t_2, t_3, t_4)=(0,1,0,1)$, then we have
$\boldsymbol{a_2=a_3}$ for  \eqref{l1-f25};
$\boldsymbol{a_2=a_1+a_3}$ for \eqref{l1-f26}; $a_3+a_2=0$ for
\eqref{l1-f27}, which is impossible, since $a_1, a_3$ are
$(-1)$-independent.

\vskip1ex The last case  $(t_1, t_2, t_3, t_4)=(0,1,1,1)$ gives:
$a_2=0$ for \eqref{l1-f25}, which implies the triviality of the
product; $\boldsymbol{a_2=a_3}$ for \eqref{l1-f26};
$a_1+a_2+a_3=0$ for \eqref{l1-f27}, which is impossible, since
$a_1, a_2, a_3$ are $(-1)$-independent.

\vskip1ex If $\boldsymbol{a_2=a_3}$, the product \eqref{l1-f0}
becomes $\big[[x_{a_1},x_{a_2}], [x_{a_2},x_{-a_1-a_2}],
x_b\big]$. It was already considered above.

\vskip1ex It remains the case $\boldsymbol{a_2=a_1+a_3}$. We
substitute $a_4=-a_1-a_3, a_2=a_1+a_3$ in \eqref{l1-f0}, represent
\eqref{l1-f0} as the sum of \eqref{l1-f1} and \eqref{l1-f2} and
expand the inner brackets in the first inner product of
\eqref{l1-f2} by the Jacobi identity:
$$\big[[x_{a_1},
x_{a_2}],[x_{a_3},x_{a_4}], x_b \big]=\big[[x_{a_1},
x_{a_1+a_3}],[x_{a_3},x_{-a_1-a_3}], x_b \big]=
$$
\begin{equation}\label{l1-f28}-\Big[\big[x_b, [x_{a_1},
x_{a_1+a_3}]\big],[x_{a_3},x_{-a_1-a_3}]\Big]
\end{equation}
$$+\big[[x_b, x_{a_3},
\boldsymbol{x_{-a_1-a_3}}],[x_{a_1},x_{a_1+a_3}]\big]
$$
\begin{equation}\label{l1-f29}-\big[[x_b, x_{-a_1-a_3}, x_{a_3},],[x_{a_1},x_{a_1+a_3}]\big].
\end{equation}
In the second summand we move the element
$\boldsymbol{x_{-a_1-a_3}}$ to the right over
$[x_{a_1},x_{a_1+a_3}]$ by the Jacobi identity:
$$\big[[x_b, x_{a_3},
\boldsymbol{x_{-a_1-a_3}}],[x_{a_1},x_{a_1+a_3}]\big]$$
$$=\big[[x_b, x_{a_3}],[x_{a_1},x_{a_1+a_3}],
\boldsymbol{x_{-a_1-a_3}}\big].
$$
\begin{equation}\label{l1-f31}+\Big[\big[x_b, x_{a_3}],\big[[x_{a_1},x_{a_1+a_3}], \boldsymbol{x_{-a_1-a_3}}\big]\Big]
\end{equation}
It suffices to prove that the products \eqref{l1-f28},
\eqref{l1-f29}, \eqref{l1-f31} and $\big[[x_b,
x_{a_3}],[x_{a_1},x_{a_1+a_3}]\big]$ are trivial. The
corresponding sequences of indices  are
\begin{equation}\label{l1-f32}\big(b, (2a_1+a_3), a_3,
(-a_1-a_3)\big),
\end{equation}
\begin{equation}\label{l1-f35} \big((b-a_1-a_3), a_3, a_1,
(a_1+a_3)\big),
\end{equation}
\begin{equation}\label{l1-f34}  \big(b, a_3, (2a_1+a_3),
(-a_1-a_3)\big),
\end{equation}
\begin{equation}\label{l1-f33} \big(b, a_3, a_1, (a_1+a_3)\big).
\end{equation}
As in the previous arguments it suffices to consider  the tuples
$(0,1,1,0)$, $(0,1,0,1)$, $(0,1,1,1)$.

 \vskip1ex If $(t_1,
t_2, t_3, t_4)=(0,1,1,0)$, then the conditions of non-triviality
of the corresponding product are: $2a_1+2a_3=0$ in \eqref{l1-f32}
and \eqref{l1-f34}; $a_3+a_1=0$ in \eqref{l1-f35} and
\eqref{l1-f33}. Both equations are impossible since $a_1$ and
$a_2$ are $(-1)$-independent and $n$ is odd.

\vskip1ex If $(t_1, t_2, t_3, t_4)=(0,1,0,1)$, then the conditions
of non-triviality of the corresponding product are: $a_1=0$ in
\eqref{l1-f32} and  \eqref{l1-f34} which implies the triviality of
the corresponding commutator; $\boldsymbol{2a_3+a_1=0} $ in
\eqref{l1-f35} and \eqref{l1-f33}.

\vskip1ex If $(t_1, t_2, t_3, t_4)=(0,1,1,1)$, then the conditions
of non-triviality of the corresponding product are: $a_1+a_3=0$ in
\eqref{l1-f32}, \eqref{l1-f34}; $2a_1+2a_3=0$ in \eqref{l1-f35},
\eqref{l1-f33}. The both conditions are  impossible since $a_1$,
$a_3$ are $(-1)$-independent and $n$ is odd.

Thus, the product \eqref{l1-f0} can be  non-trivial only if
$a_1=-2a_3$. Then  $a_2=a_1+a_3=-a_3$, $a_4=-a_1-a_3=a_3$ and
$$\Big[[x_{a_1}, x_{a_2}],[x_{a_3}, x_{a_4}]\Big]=
\Big[[x_{-2a_3}, x_{-a_3}],[x_{a_3},
x_{a_3}]\Big]=
$$
$$=\Big[x_{-2a_3},[x_{a_3}, x_{a_3}], x_{-a_3}
\Big]+\Big[x_{-2a_3}, \big[x_{-a_3},[x_{a_3},
x_{a_3}]\big]\Big].$$ The first summand is trivial since it
contains an initial segment with zero sum of indices. The inner
product $\big[x_{-a_3},[x_{a_3}, x_{a_3}]\big]$ of the second
summand can be represented as the sum of two products of the form
$ \big[x_{-a_3}, x_{a_3}, x_{a_3}\big]$ by expanding the brackets
by the Jacobi identity. As $ \big[x_{-a_3}, x_{a_3}, x_{a_3}\big]$
also contains an initial segment with zero sum of indices, the
second summand is trivial and, consequently, \eqref{l1-f0} is
trivial as well.
\end{proof}

\vskip2ex

\begin{lemma}\label{l2} Suppose that a $(\Z/n\Z)$-graded Lie algebra $L$ with $L_0=0$, $n$
odd, satisfies the selective metabelian condition~\eqref{select}.
Let $c,a,b\in \Bbb Z/n\Bbb Z$, $n$ odd,
 $o(a)>3$. Then (under the Index
Convention)
$$\Big[x_{c}, \underbrace{x_{a}, x_{a}, \ldots, x_{a}}_{7}\Big]=0$$
for all $x_c\in L_c$, $x_a=[x_{b}, x_{a-b}]\in [L_b, L_{a-b}]$.

\end{lemma}

Note that under the Index Convention the elements $x_a$ in the
product above can be different.

\begin{proof} We denote $[x_{c}, \underbrace{x_{a}, x_{a}, \ldots, x_{a}}_{k}]$,
$k=0,1,\ldots,5$, by $x_{c+ka}$ under the Index Convention.
Suppose by contradiction that $$X=\big[x_{c}, \underbrace{x_{a},
x_{a}, \ldots, x_{a}}_7\big] \neq 0.$$ It follows that for all
$k=0,1,2,\ldots, 5$ the products
\begin{equation}\label{l2-f0} \big[x_{c}, \underbrace{x_{a}, x_{a},
\ldots,
x_{a}}_k, x_a, x_a\big]=\big[[x_{c+ka}, x_a],\, [x_{b},\, x_{a-b}]
\big] \end{equation}
 are not trivial as well. Hence, the
corresponding sequences of indices
$$\big((c+ka),a,b,(a-b)\big)$$ for $k=0,1,2,\ldots, 5$
are all $(-1)$-dependent, i.e. for each $k=1,2,\ldots, 5$  there
are $t_1, t_2, t_3, t_4,\in \{0,1\}$ (depending on $k$) such that
$(t_1,t_2, t_3, t_4)\neq (0,0,0,0)$ and
$$t_1(c+ka)+t_2a+t_3b+t_4(a-b)=0.$$

Suppose  first that $\big(a, b, (a-b)\big)$ is $(-1)$-independent,
i.e. $t_1=1$.
 If $(t_1, t_2, t_3, t_4)=(1,0,0,0), (1,1,0,0)$, (1,0,1,1) or
$(1,1,1,1)$, then \eqref{l2-f0} is trivial since it  contains an
inner product with zero sum of indices. If  $(t_1,t_2, t_3,
t_4)=(1,0,1,0)$, then $c+ka=-b$. If $(t_1,t_2, t_3,
t_4)=(1,0,0,1)$, then $c+ka=-a+b$. If $(t_1, t_2, t_3,
t_4)=(1,1,1,0)$, then $c+ka=-a-b$. If $(t_1,t_2, t_3, t_4)=(1,
1,0,1)$, then $c+ka=-2a+b$.

\vskip1ex

 It follows that for $k=0,1,\ldots, 5$ $c+ka\in \{(-b),(-a+b), (-a-b),
 (-2a+b)\}$. It is impossible, since $o(a)$ is actually greater than
 4 (as 2 does not divide $n$) and $c+ka$ takes at least 5 different
 values for $k=0,\ldots, 5$.






\vskip1ex

\vskip1ex Let now $a, b, a-b$ be $(-1)$-dependent. By definition
there are $t_1, t_2, t_3\in \{0,1\}$ such that
$t_1a+t_2b+t_3(a-b)=0$ and $(t_1, t_2, t_3)\neq (0,0,0).$ We
consider all possible 3-tuple $(t_1, t_2, t_3)$. Each of the
tuples $(1,0,0), (0,1,0), (0,0,1), (0,1,1), (1,1,1)$ implies the
triviality of \eqref{l2-f0}. For $(1,0,1)$ we obtain
$\boldsymbol{2a-b=0}$ and for $(1,1,0)$ $\boldsymbol{a+b=0}$.

\vskip1ex Consider the case $\boldsymbol{b=2a}$. Then
$x_{a}=[x_{2a}, x_{-a}]$ and
$$\Big[[x_{c+ak}, x_{a}], [x_{b},x_{a-b}]\Big]=\Big[x_{c+ak}, [x_{2a}, x_{-a}], [x_{2a},x_{-a}]\Big]
=$$
$$ \Big[[x_{c+ak}, x_{2a}, \boldsymbol{x_{-a}}],
[x_{2a},x_{-a}]\Big]$$
 \begin{equation}\label{l1-f36}-\Big[[x_{c+ak}, x_{-a}, x_{2a}], [x_{2a},x_{-a}]\Big].
 \end{equation}
 In the
first summand we move $\boldsymbol{x_{-a}}$ to the right over the
$[x_{2a},x_{-a}]\in L_{a}$ to get
\begin{equation}\label{l1-f37}\Big[\big[[x_{c+ak}, x_{2a}],
[x_{2a},x_{-a}]\big], x_{-a}\Big].
\end{equation}
(The  additional term, which arises by the Jacobi identity, is
trivial as it contains an inner product
$\big[[\boldsymbol{x_{-a}},[x_{2a},x_{-a}]\big]\in L_{0}$.) We
determine under which conditions \eqref{l1-f36} or \eqref{l1-f37}
could be non-trivial. The sequence $\big((c+ka-a), 2a,
2a,(-a)\big)$ or $\big((c+ka), 2a, 2a,(-a)\big)$ should be
(-1)-dependent.  As in the previous arguments, considering  all
possible tuples $(t_1, t_2, t_3, t_4)$, $t_i\in\{0,1\}$, and
equations
$$t_1(c+ka-a)+t_2\cdot 2a+t_3\cdot 2a+t_4\cdot (-a)=0,
\,\,\,\,\,\,\,\,\,\,t_1(c+ka)+t_2\cdot 2a+t_3\cdot 2a+t_4\cdot
(-a)=0,
$$ we obtain that $$c+ka\in \{(-4a),(-3a),(-2a), (-a), a, 2a\}$$ for $t_1=1$   and $3a=0$ for
$t_1=0$. The case of $o(a)=3$ is impossible by assumption. If
$c\in \{(-4a),(-3a),(-2a),(-a)\}$, then the product
$$X=\big[x_c, \underbrace{x_a,\ldots, x_a}_7\big]$$ is trivial as it
 contains an inner product with zero sum of indices. If $c=a$ or
$c=2a$, then respectively  $c+2a$ or $c+a$ is equal to $3a$, which
belongs to the set $$\{(-4a),(-3a),(-2a), (-a), a, 2a\}.$$ It
follows that $o(a)=5$ or $7$. In this case the product $X$ is also
trivial since there is an initial segment with zero sum of
indices.

\vskip1ex
 The case where $\boldsymbol{b_1=-a}$ is actually the same as the
previous one, since $x_{a}=[x_{-a}, x_{2a}]=-[x_{2a}, x_{-a}]$.

\vskip1ex We have proved that in all the cases $X$ is trivial that
contradicts our assumption.
\end{proof}


\begin{lemma}\label{l3} Suppose that a $(\Z/n\Z)$-graded Lie algebra $L$ with $L_0=0$, $n$
odd,  satisfies the following selective condition (under the Index
Convention):
\begin{equation}\label{select2}
\big[[x_{d_1},x_{d_2}], [x_{d_3}, x]\big]=0\quad \text{ whenever }
(d_1,d_2,d_{3}) \text{ is $(-1)$-independent.}
\end{equation}
 Then  $L$
is metabelian.
\end{lemma}

\begin{proof} Let $a, b, c,d\in \Bbb Z/n\Bbb Z$, $x_a\in L_a, x_{b}\in L_b, x_c\in L_c, x_d\in
L_d$. It suffices to prove that
$$Y=\big[[x_{a},x_{b}],[x_{c},x_{d}]\big]= 0.$$
 Suppose by contradiction that $Y\neq 0.$  Note that if $(a,b)$ or $(c,d)$ is $(-1)$-dependent, then
respectively $a=-b$ or $c=-d$. It follows that  $Y=0,$ since it
contains a subproduct  with zero  sum of indices. Therefore, we
can assume that $(a,b)$ and $(c,d)$ are $(-1)$-independent.
 If $c$ or $d$ does not belong to
$D(a, b)$, then by  \eqref{select2} the product $Y$ is trivial.
Thus, we can suppose that $c,d\in D(a, b)$. This means that $c,
d\in\{-a, -b, -a-b\}$.

\vskip1ex

 We will consider all possible cases for the pair
$(c,d)$.

 \vskip1ex
{\bf Case: $\boldsymbol{c=-a}$, $\boldsymbol{d=-a}$}. If the
sequence $\big((-a), (-a), b\big)$ is $(-1)$-independent, then
$Y=0$ by \eqref{select2}.  Suppose that $\big((-a), (-a), b\big)$
is $(-1)$-dependent. It follows  that $\boldsymbol{b=a}$ or
$\boldsymbol{b=2a}$. In the  case where $b=a$ we have
$$Y=
\big[[x_{a},x_{a}],[x_{-a},x_{-a}]\big]=0,
$$
 and in the case where $b=2a$ we have
\begin{equation}\label{l3-1}
Y=\big[[x_{a},\boldsymbol{x_{2a}}],[x_{-a},x_{-a}]\big]=
\big[x_{a},[x_{-a},x_{-a}], \boldsymbol{x_{2a}}\big] +\Big[x_{a},
\big[\boldsymbol{x_{2a}},[x_{-a},x_{-a}]\big]\Big].
\end{equation}
Expanding the inner brackets of the subproducts
$\big[x_{a},[x_{-a},x_{-a}]\big]$ and
$\big[\boldsymbol{x_{2a}},[x_{-a},x_{-a}]\big]$ by the Jacobi
identity in the above terms, we obtain a linear combinations of
products, each of them contains an initial segment with zero sum
of indices. Hence, \eqref{l3-1} is trivial.

\vskip1ex {\bf Case: $\boldsymbol{c=-a}$, $\boldsymbol{d=-a-b}$}.
If $\big((-a), (-a-b), b\big)$ is $(-1)$-independent, then $Y=0$
by \eqref{select2}.  If $\big((-a), (-a-b), b\big)$ is
$(-1)$-dependent, then $b=a$ (recall that $\big(c,d)=((-a),
(-a-b)\big)$ is $(-1)$-independent).
 Thus,
$$Y=\big[[x_{a},x_{a}],[x_{-a},x_{-2a}]\big]=
-\big[[x_{-a},x_{-2a}],[x_{a},x_{a}]\big].
$$
If we replace $a$ by $-a$, we get the product \eqref{l3-1}
considered above. Thus, $Y$ is trivial  in this case as well.

\vskip1ex {\bf Case: $\boldsymbol{c=-a-b}, \boldsymbol{d=-a-b}$}.
As in the previous cases we can suppose that $\big((-a-b), (-a-b),
b\big)$ is $(-1)$-dependent. Then $b=-2a$ and
$$Y=\big[[x_{a},x_{-2a}],[x_{a},x_{a}]\big].$$ Expanding the inner
brackets by the Jacobi identity we get a linear combination of
products with initial segments from $L_0$. Therefore, $Y$ is zero
as well.

 \vskip1ex

 The case of $\boldsymbol{c=-a}, \boldsymbol{d=-b}$ is trivial,
 since the sum of indices in $Y$ is equal to zero in this case.

\vskip1ex
 The remaining cases are the same as treated above by
interchanging $a$, $b$ or $c$ and $d$.

\end{proof}

\begin{lemma}\label{l4}
Suppose that a $(\Z/n\Z)$-graded Lie algebra $L$ with $L_0=0$, $n$
odd, satisfies the selective metabelian condition~\eqref{select}.
Let $u_{d_1},u_{d_2},u_{d_3}, u_{d_4}$  be homogeneous elements of
$L$ (under the Index Convention), such that  the sequence of
indices $(d_1,d_2, d_3)$ is $(-1)$-independent and $d_4\in D(d_1,
d_2, d_3)$.  Then (under the Index Convention) every product of
the form
\begin{equation}\label{eq4}
\big[[u_{d_1},u_{d_2}],[u_{d_3},u _{d_4}],
x_{i_1},\dots,x_{i_t}\big],
\end{equation}
where the $x_{i_k}=[x_{w_k}, x_{i_k-w_k}]$, $k=1,\ldots, t$, are
homogeneous products of length~2,
 can be
written as a linear combination of products of the form
\begin{equation}\label{eq5}
\big[[u_{d_1},u_{d_2}],[u_{d_3},u_{d_4}],
m_{j_1},\dots,m_{j_{s}}\big],
\end{equation}
where $s\leq t$, $m_{j_k}=[m_{u_{k}}, m_{j_{k}-u_k}]$ such that
 $j_k \in \widetilde D (d_1, d_2, d_3, d_4)$ and at least one of the indices $u_{k}$ or
$j_{k}-u_k$ belongs to $\widetilde D (d_1, d_2, d_3, d_4)$ as
well.
\end{lemma}

\begin{proof}

\vskip1ex Set $U=\big[[u_{d_1},u_{d_2}],[u_{d_3},u_{d_4}]\big]$.
We use induction on $t$. If $t=0$, \eqref{eq4} has the required
form. If $t=1$ and $i_1\notin \widetilde D (d_1, d_2, d_3, d_4)$,
then \eqref{eq4} is equal to zero by Lemma~\ref{l1}. Suppose that
$i_1\in \widetilde D (d_1, d_2, d_3, d_4)$. If either $w_1$ or
$i_1-w_1$ belongs to $\widetilde D (d_1, d_2, d_3, d_4)$, the
product \eqref{eq4} has the required form.  Otherwise, we expand
the brackets of the subproduct $[x_{w_1},x_{i_1-w_1}]$ by the
Jacobi identity:
$$
\big[U, x_{i_1}\big]=
 \big[U,
 [x_{w_1},x_{i_1-w_1}]\big]=
\big[U, x_{w_1},x_{i_1-w_1}\big]-\big[U, x_{i_1-w_1},
x_{w_1}\big].
$$
The both summands are equal to zero by Lemma \ref{l1}.

\vskip1ex Assume that $t>1$. If all the indices $i_k$ belong to
$\widetilde D(d_1, d_1, d_3, d_4)$ and for each $k$ either $w_k$
or $i_k-w_k \in \widetilde D(d_1, d_1, d_3, d_4)$, then the
product \eqref{eq4} is of the required form with $s=t$. Suppose
that in \eqref{eq4} there is an element $x_{i_k}$ such that
\begin{itemize}
\item   either $i_k\notin \widetilde D(d_1, d_1, d_3, d_4)$
 \item  or $i_k\in \widetilde D(d_1, d_1, d_3, d_4)$, but both $w_k,
i_k-w_k$ do not belong to $ \widetilde D(d_1, d_1, d_3, d_4)$.
\end{itemize}
 We
choose such an element with $k$ as small as possible and use $k$
as a second induction parameter to prove that \eqref{eq4}   can be
represented
 as a linear combination of
the products  of the  form \eqref{eq5} and  a linear combination
of products in which the number of the ``good'' elements after $U$
is equal to $k$, i.e. the   products of the form
\begin{equation}\label{eq51}
\big[U, m_{j_1},\dots, m_{j_k}, x_{i_{k+1}},\ldots x_{i_{t}}\big],
\end{equation}
where for $l=1,\ldots, k$ $m_{j_l}=[m_{u_{l}}, m_{j_{l}-u_l}]$
with
 $j_l \in \widetilde D (d_1, d_2, d_3, d_4)$ and at least one of the indices $u_{l}$ or
$j_{l}-u_l$ belongs to $\widetilde D (d_1, d_2, d_3, d_4)$.

\vskip1ex If $k=1$ and $i_1 \notin \widetilde D(d_1, d_1, d_3,
d_4)$, then the product \eqref{eq4} is zero by Lemma~\ref{l1}.
 If $i_1 \in \widetilde D(d_1, d_1, d_3, d_4)$, but $w_1, i_1-w_1\notin \widetilde D(d_1, d_1, d_3, d_4)$,
we expand the brackets of the subproduct $[x_{w_1},x_{i_1-w_1}]$
in \eqref{eq4} by the Jacobi identity (as  for $t=1$) and obtain
the sum of two products that are trivial by Lemma~\ref{l1}. Hence,
$i_1 \in \widetilde D(d_1, d_1, d_3, d_4)$ and  at least one of
the indices $w_1$ or $i_1-w_1$ belongs to $\widetilde D (d_1, d_2,
d_3, d_4)$. Thus, the product \eqref{eq4} has at least 1 element
of the required form after $U$.

\vskip1ex
  Suppose
that $k\geq 2$ and write
$$\big[U, \dots,x_{i_{k-1}}, x_{i_{k}}\ldots
\big]= \big[U,\dots,x_{i_k},x_{i_{k-1}},\dots\big]+
$$
$$
\big[U,\dots,[x_{i_{k-1}},x_{i_k}],\dots\big].
$$
By the induction hypothesis, the second summand and the initial
segment $[U,\ldots, x_{i_{k-2}}, x_{i_k}]$ of the fist summand
have the required form \eqref{eq5} because they are  shorter than
\eqref{eq4}. Since $x_{i_{k-1}}$ is already ``good'' by our
assumption,  the first term is equal to a linear combination of
 products
of the form \eqref{eq51} in which the number of the ``good''
elements is $k$ as required.

\vskip1ex Since the length of the products does not increase, but
the number of ``good'' elements grows, the induction process will
end with the representation of the original product in  the
required form.
\end{proof}

Let $d_1, d_2, d_3, d_4\in \Bbb Z/n \Bbb Z$. Recall that
$\widetilde D (d_1, d_2, d_3, d_4)$ is the set of all linear
combination of the form $\alpha_1 d_1+\alpha_2 d_2+\alpha_3
d_3+\alpha_4 d_4$ with $\alpha_1\in \{0,\pm 1,\pm 2\}$.  Set
$\widetilde D =5^4$. Note that the number of elements in the set
$\widetilde D(d_1, d_2,d_3,d_4)$ is at most $\widetilde D $ for
all $d_1, d_2, d_3, d_4\in \Bbb Z/n \Bbb Z$.

\begin{lemma}\label{l5} Suppose that a $(\Z/n\Z)$-graded Lie algebra $L$ with $L_0=0$, $n$
odd, satisfies the selective metabelian condition~\eqref{select}.
Let  $u_{d_1},u_{d_2},u_{d_3}, u_{d_4}$  be  homogeneous elements
of $L$ (under the Index Convention) with the $(-1)$-independent
sequence of indices $(d_1,d_2, d_3)$ and $d_4\in D(d_1, d_2,
d_3)$.  The ideal of $[L,L]$ generated by
$$U=\big[[u_{d_1},u_{d_2}],[u_{d_3},u_{d_4}]\big]$$
 is spanned by
products of the form
\begin{equation}\label{eq6}
\big[U, m_{i_1},\dots,m_{i_u},m_{i_{u+1}},\dots,m_{i_{v}}\big],
\end{equation}
where $m_{i_k}\in [L,L]\cap L_{i_k}$, $u\leq 6\widetilde D^2$,
$i_k\in \widetilde D (d_1,d_2, d_3, d_4)$, such that $o(i_k)>3$
for $k\leq u$, and $o(i_k)=3$ for $k>u$.
\end{lemma}

\begin{proof}
Let $R$ be the span of all products of the form \eqref{eq6}. The
ideal of $[L,L]$ generated by $U$ is spanned by products of the
form  $[U, x_{i_1},\dots,x_{i_{t}}]$, where each
$x_{i_k}=[x_{w_{k}}, x_{i_{k}-w_k}]$ (under the Index Convention).
By Lemma \ref{l4} each of such product can be represented as a
linear combination of products of the form \eqref{eq5}:
$$
W=\big[[u_{d_1},u_{d_2}],[u_{d_3},u_{d_4}],
m_{j_1},\dots,m_{j_{s}}\big],
$$
where $s\leq t$ and $m_{j_k}=[m_{u_{k}}, m_{j_{k}-u_k}]$ such that
$j_k \in \widetilde D (d_1, d_2, d_3, d_4)$, and either $u_{k} \in
\widetilde D (d_1, d_2, d_3, d_4)$ or $j_{k}-u_k \in \widetilde D
(d_1, d_2, d_3, d_4)$. It is sufficient to show that $W$ belongs
to~$R$.

\vskip1ex We use induction on $s$. If $s=1$, it is clear that
$W\in R$, so we assume that $s>1$. We will prove that  $W$ does
not change modulo $R$ under any permutation of the~$m_{j_k}$.
 Write
\begin{equation}\label{eq61}
W= [U,m_{j_1},\dots,m_{j_v}, m_{j_{v-1}},\dots,m_{j_s}]+
[U,m_{j_1},\dots,[m_{j_{v-1}},m_{j_{v}}],\dots, m_{j_s}].
\end{equation}
 By Lemma~\ref{l4} the second summand is a linear combination of
products of the form \eqref{eq5} each of which is shorter than
$W$. Thus,  the second summand in \eqref{eq61} belongs to $R$ by
induction. It follows that we can freely permute $m_{j_k}$  modulo
$R$.

\vskip1ex For each pair of indices $j_k, u_k\in \widetilde D (d_1,
d_2, d_3, d_4)$, such that  $o(j_k)>3$, we place the elements
$m_{j_k}=[m_{u_k}, m_{j_k-u_k}]$ or $m_{j_k}=[m_{j_k-u_k},
m_{u_k}]$ (if there are any) with the same unordered pair of
indices $u_k, j_k-u_k$ next to each other. If for  some pair there
are at least~7 of them, the product $W$ is trivial by Lemma
\ref{l2}. Suppose  that for each pair $(j_k, u_k)$, all such
$m_{j_k}$ occur less than $7$ times. We place all these elements
right after the~$U$. Since the number of pairs $(i_k, u_k)\in
\widetilde D (d_1, d_2, d_3, d_4)^2$ is at most $\widetilde D^2$,
the initial segment has length  at most $6\widetilde D^2$, so the
resulting product takes the required form~\eqref{eq6}.
\end{proof}

\begin{corollary}\label{c6}
Suppose that a $(\Z/n\Z)$-graded Lie algebra $L$ with $L_0=0$, $n$
odd, satisfies the selective metabelian condition~\eqref{select},
and let $M=[L,L]$. Suppose that $(d_1,d_2, d_3)$ is an
$(-1)$-independent sequence and $d_4\in D(d_1, d_2, d_3)$. Then
the ideal in $[L,L]$ generated by $ \big[[L_{d_1}, L_{d_2}],
[L_{d_3}, L_{d_4}]\big]$  has at most $3\widetilde D^{u}$
non-trivial components of the induced grading, where $ \widetilde
D= 5^4$, $u \leq 6 \widetilde D^2=6\cdot 5^8$.
\end{corollary}

\begin{proof}
Let $U=\big[[u_{d_1}, u_{d_2}], [u_{d_3}, u_{d_4}]\big]$ be an
arbitrary homogeneous product in $L$ with the given indices (under
the  Index Convention)representation of the original product into  the required form and let $I=I(u_{d_1}, u_{d_2}, u_{d_3},
u_{d_4})$ be the ideal in $[L,L]$ generated by $U$. By Lemma
\ref{l5} an element of  $I$
 can be represented as a
linear combination of the products of the form \eqref{eq6}. Since
in \eqref{eq6} we have ${i_k}\in \widetilde D(d_1, d_2, d_3, d_4)$
for all $k=1,\dots,u$, the sum of all indices of the initial
segment $\big[U, m_{i_1}, \dots, m_{i_u}\big]$ can take at most
$\widetilde D^{u}$ values. Let $B$ be the set of elements of order
$3$ in $\Bbb Z/n\Bbb Z$. This is an additive subgroup of order 3.
Clearly, the sum of the remaining indices in \eqref{eq6} belongs
to $B$. It follows that the sum of all indices in \eqref{eq6} can
take at most $3\widetilde D^{u}$ values.   By Lemma~\ref{l5} the
number $u$ is at most $6 \widetilde D^2=6\cdot 5^8$. So $I$ has
the number of nontrivial components bounded by a constant
$3\widetilde D^{u}$, where $\widetilde D= 5^4$, $u\leq 6\cdot
5^8$.

\vskip1ex It follows from the proofs of Lemmas~\ref{l4} and
\ref{l5} that the set of indices of all possible non-trivial
components in $I$ is completely determined by the tuple
$(d_1,d_2,d_3, d_4)$ and does not depend on the choice of
$U=\big[[u_{d_1}, u_{d_2}], [u_{d_3}, u_{d_4}]\big]$. Since the
ideal in $[L,L]$ generated by $\big[[L_{d_1}, L_{d_2}], [L_{d_3},
L_{d_4}]\big]$ is the sum of ideals $I\big(u_{d_1}, u_{d_2},
u_{d_3}, u_{d_4}\big)$ over all possible $u_{d_1},u_{d_2},u_{d_3},
u_{d_4} $, the result follows.
\end{proof}

\begin{lemma}\label{l7}
Suppose that a homogeneous ideal $T$ of a Lie algebra $L$ has only
$e$ non-trivial components. Then $L$ has at most $e^2$ components
that do not centralize $T$.
\end{lemma}
\begin{proof}
Let $T_{i_1},\dots,T_{i_e}$ be the non-trivial homogeneous
components of $T$ and let $S=\{i_1,\dots,i_e\}$. Suppose that
$L_i$ does not centralize $T$. Then for some $j\in S$ we have
$i+j\in S$. So there are at most $|S|\times|S|$ possibilities for
$i$, as required.
\end{proof}

\begin{proposition}\label{razresh}
Suppose that a $(\Z/n\Z)$-graded Lie algebra $L$ with $L_0=0$, $n$
odd, satisfies the selective metabelian condition~\eqref{select}.
Then $L$ is solvable of derived length bounded by a constant.
\end{proposition}

Note that $L$ is solvable of $n$-bounded derived length by
Kreknin's theorem, but we need a constant bound that does not
depend of $n$.

\begin{proof}
Let $J$ be the ideal of $M=[L,L]$ generated by all products
$\big[[M_{i_1},M_{i_2}],[M_{i_3}, M_{i_4}]\big]$, where
$(i_1,i_2,i_3)$ ranges through all $(-1)$-independent sequences of
length~$3$ and  $i_4$ ranges through all elements of $D(i_1, i_2,
i_3)$. The induced $(\Z/n\Z)$-grading of $M/J$ has trivial
zero-component and $M/J$ satisfies the following selective
condition (under the Index Condition):
$$
\big[[x_{i_1}, x_{i_2}], [x_{i_3}, x_j]\big]=0
$$
whenever $(i_1, i_2,i_3)$ is $(-1)$-independent sequence. Indeed,
if $j\notin D(i_1, i_2, i_3)$, then $(i_1, i_2, i_3, j)$ is
$(-1)$-independent sequence and the product is trivial by
selective metabelian condition~\eqref{select}; otherwise  it
belongs to $J$. By Lemma \ref{l3} $M/J$ is metabelian, that is,
$L^{(3)}\leq J$.

\vskip1ex Consider now an arbitrary $(-1)$-independent sequence
$(i_1,i_2,i_3)$ and $i_4\in D(i_1, i_2, i_3)$. The ideal $T$  in
$[L,L]$ generated by $\big[[M_{i_1},M_{i_2}],[M_{i_3},
M_{i_4}]\big]$ has at most $e=3\widetilde D ^u$ (where, recall,
$\widetilde D= 5^4$, $u\leq 6\widetilde D^2$)  non-trivial grading
components by Corollary \ref{c6}. Therefore, by Lemma \ref{l7},
there are at most $e^2$ components that do not centralize $T$.
Since $C_M(T)$ is also a homogeneous ideal, it follows that the
quotient $M/C_M(T)$ has at most $e^2$ non-trivial components.
Since the induced $(\Z/n\Z)$-grading of $M/C_M(T)$ also has
trivial zero component, by Shalev's generalization~\cite{shalev}
of Kreknin's theorem we conclude that $M/C_M(T)$ is solvable of
$e$-bounded derived length, say, $f_1$. Therefore, $M^{(f_1)}$,
the corresponding term of the derived series, centralizes $T$.
Since $f_1$ does not depend on the choice of the
$(-1)$-independent seuence $(i_1,i_2,i_3)$ and $i_4\in
D(i_1,i_2,i_3)$, and $J$ is the sum of all such ideals $T$, it
follows that $[M^{(f_1)},J]=0$. Recall that $L^{(3)}\leq J$.
Hence, $[L^{(f_1+1)},L^{(3)}]=0$. Thus, $L$ is solvable of derived
length at most $\max\{3,f_1+1\}+1$.
\end{proof}

As shown in section \ref{section-reduction}, to prove the theorem
\ref{th1} it is enough to show the solvability of bounded length
of a $(\Z/n\Z)$-graded Lie algebra with trivial zero-component
satisfying the selective metabelian condition. Thus, Proposition
\ref{razresh} implies  Theorem \ref{th1}.


\vskip1ex
\begin{thebibliography}{29}

\bibitem{khu-ma-shu} {\sc E.\,I.\,Khukhro, N.\,Yu.\,Makarenko,
P.\, Shumyatsky}, Frobenius groups of automorphisms and their
fixed points, {\it Forum Mathematicum}, {\bf 26} (2014),
73--112--12. https://doi.org/10.1515/form.2011.152

\bibitem{ma-shu} {\sc N.\,Yu.\,Makarenko,
P.\, Shumyatsky}, Frobenius groups as groups of automorphisms,
{\it Proc. Amer. Math. Soc.}, {\bf 138} (2010), 3425-3436. DOI:
https://doi.org/10.1090/S0002-9939-2010-10494-X




\bibitem{shalev} A. Shalev, Automorphisms of finite groups of bounded rank,
\emph{Israel J. Math.} \textbf{82} (1993), 395--404.

\bibitem{kr} V. A. Kreknin, The solubility of Lie
algebras with regular automorphisms of finite period, {\it Dokl.
Akad. Nauk SSSR} {\bf 150} (1963), 467--469 (in Russian); English
transl., {\it Math. USSR Doklady} {\bf 4} (1963), 683--685.

\bibitem{kh-book} E. I. Khukhro, \emph{Nilpotent Groups and their Automorphisms},
de\,\,Gruyter--Verlag, Berlin, 1993.


\bibitem{hall} P. Hall, Some sufficient conditions for a group to be
nilpotent, \emph{Illinois J. Math.} {\bf 2} (1958), 787--801.



\end{thebibliography}
\end{document}